\documentclass{llncs}

\usepackage{makeidx}  
\usepackage{amsmath,amssymb}
\def\smallddots{\mathinner{\raise7pt\hbox{.}\raise4pt\hbox{.}\raise1pt\hbox{.}}}
\def\smallsdots{\mathinner{\raise1pt\hbox{.}\raise4pt\hbox{.}\raise7pt\hbox{.}}}
\usepackage{multirow}

\spnewtheorem{algorithm}{Algorithm}{\bfseries}{\itshape}
\spnewtheorem{fact}{Fact}{\bfseries}{\itshape}
\spnewtheorem{procedure}{Procedure}{\bfseries}{\itshape}
\spnewtheorem{subroutine}{Subroutine}{\bfseries}{\itshape}
\spnewtheorem{flowchart}{Flowchart}{\bfseries}{\itshape}

\begin{document}
\title{\bf  Polynomial Root Isolation by Means of Root Radii Approximation
%\thanks{%V. Y. Pan and L. Zhao 
%This work has been supported by NSF Grant CCF 1116736 
%and  PSC CUNY Award 67699-00 45.}
}
\author{Victor Y. Pan 
 $^{[1, 2],[a]}$
and
 Liang Zhao
$^{[2],[b]}$
}

%\author{Victor Y. Pan}   % abbreviated author list (for running head)
%
\institute{
%$^{[1]}$
Departments of Mathematics and Computer Science \\
Lehman College and the Graduate Center of the City University of New York \\
Bronx, NY 10468 USA \\
\and
%$^{[2]}$ 
Ph.D. Programs in Mathematics  and Computer Science \\
The Graduate Center of the City University of New York \\
New York, NY 10036 USA \\
$^{[a]}$
\email{victor.pan@lehman.cuny.edu},\\ home page:
\texttt{http://comet.lehman.cuny.edu/vpan/ }\\
%\and\\
$^{[b]}$ 
lzhao1@gc.cuny.edu \\
}
% \date{}

%------------------------------------------------------------------------------

\maketitle

%------------------------------------------------------------------------------

\begin{abstract}
Univariate polynomial root-finding is a classical subject, still 
important for modern computing. Frequently one seeks
just the real roots of a real coefficient polynomial.
They can be approximated at a low computational cost if the
polynomial has no nonreal roots, but for high degree 
polynomials, nonreal roots are typically 
much more numerous than the real ones. The challenge is known 
for long time, and the subject has been
 intensively studied.
The Boolean cost bounds for the 
refinement of the simple and isolated real roots have been
decreased to nearly optimal, but 
the success has been more limited 
at the stage of  
the isolation of real roots.
 We obtain substantial progress  
by  applying the
algorithm of \cite{S82} for the approximation of the root radii,
that is, the distances of the roots to the origin.
Namely we isolate the simple and well conditioned real roots 
of a polynomial at the  
 Boolean cost dominated by the nearly optimal bounds
for the refinement of such roots. 
We also extend our algorithm 
 to the  isolation of    
  complex, possibly multiple, roots and root clusters
staying within the same (nearly optimal) asymptotic Boolean cost bound.
Our numerical tests with benchmark polynomials
performed with the IEEE standard double precision
show that our nearly optimal real root-finder
is practically promising.
Our techniques are simple, and 
their power and application range
may increase in combination with the known efficient methods. 
\end{abstract} 

%------------------------------------------------------------------------------


\section{Introduction}\label{sintr}

Assume
a univariate polynomial of degree $n$  with
 real coefficients,
\begin{equation}\label{eqpoly}
 p(x)=\sum^{n}_{i=0}p_ix^i=p_n\prod^n_{j=1}(x-x_j),~~~ p_n\ne 0,
\end{equation}
which  has $r$ real roots 
  $x_1,\dots,x_r$ and $s=(n-r)/2$
pairs of non-real complex conjugate roots.
In many applications, e.g., to algebraic and geometric optimization,
one seeks only the real roots,
which 
make up just a small fraction of all roots. 
 This motivates a well studied subject of real root-finding
(see, e.g., \cite{PT13}, \cite{PT15}, \cite{SM15},  
and the extensive bibliography therein),
but the 
most popular packages of subroutines 
for  numerical root-finding such as MPSolve 2.0 
\cite{BF00}, Eigensolve  \cite{F02},
and MPSolve 3.0 \cite{BR14} still
approximate the $r$ real roots about as fast 
and as slow as all the $n$ complex roots.

A typical fast real root-finder 
consists of two stages. At first one isolates all 
simple and well 
conditioned
real roots (that is, the ones that
 admit isolation). Namely, one computes
some complex discs,
each covering a single real root
and no other roots of the polynomial $p(x)$.
Then 
the isolated roots
are approximated fast by means of some
specialized root refinement algorithms.
The record and nearly optimal Boolean complexity at this stage
has been obtained in the algorithm of \cite{PT13}, \cite{PT15}.

Presently we achieve progress, so far missing, at the former stage
of the real root isolation. Our algorithm 
performs this stage and consequently 
approximates all the simple and well conditioned real roots
 within the same asymptotic 
Boolean complexity bounds of the papers
 \cite{PT13}, \cite{PT15} (see our Theorem \ref{thrrrr}).

We have also extended our algorithm 
 to the  isolation of    
  complex, possibly multiple, roots and root clusters
staying within the same (nearly optimal) asymptotic Boolean cost bound
(see Section \ref{scrisl}).

As in the papers  \cite{PT13},  \cite{PT15},
we approximate only simple and well conditioned
real roots, but not the multiple and ill conditioned roots.
The latter roots little affect the complexity of our algorithm, e.g.,
our cost estimate do not depend on the minimal distance between the roots
and do not
include the terms like $\log({\rm Discr} (p)^{-1})$.

Our overall cost bound also matches the nearly optimal one of the papers
 \cite{P95} and 
 \cite{P02}. Their algorithm approximates all 
complex roots of a polynomial, 
but combines  
a number of advanced techniques. This
makes it much harder to implement and even to comprehend
than 
our algorithm, which is much less involved,
more transparent and 
more accessible for the implementation.

We have tested  our  algorithm 
applied  with the IEEE standard double precision
to some benchmark polynomials 
with small numbers of real roots. 
The  test results are 
quite encouraging and are
in good
accordance with our formal study.
The overall Boolean complexity of our isolation algorithm 
is dominated at the stage of performing 
the Dandelin's auxiliary root-squaring iterations,
whose Boolean cost grows fast when  their
number  increases,
but in our tests this number  
grew very slowly  as we increased 
the degree of the 
%univariate
 input 
polynomials 
 from 64 to 1024.

Technically, we achieve our progress by incorporating the old
algorithm of \cite{S82}, for the approximation of the root radii,
that is, the distances of the roots to the origin
(we refer the reader to  \cite{O40},
 \cite{G72}, \cite{H74}, and  \cite{B79} on 
the preceding works  related to that algorithm).
We feel that Sch{\"o}nhage in \cite{S82}  
 used only a small part of the potential power of the algorithm.
Namely he applied it to the rather modest task of the isolation of a single 
complex root,
whereas we apply this simple but surprisingly efficient tool 
to the isolation of all simple 
and well conditioned real roots.
We hope that  this tool will be efficiently incorporated into
 other root-finders as well.

Our another basic sub-algorithm  performs multi-point 
evaluation of the polynomial $p(x)$.
The algorithm of \cite{MB72},
recently used in the root-finders of 
 \cite{PT13}, \cite{PT14b}, \cite{PT15},
and  \cite{SM15},  solves this  
problem at  a low Boolean cost, 
but 
  is  numerically unstable and only works with extended precision.
Performed  with double precision, it produces corrupted output
 already for polynomials of moderate degree (say, about 50 or so).

In some cases we can avoid this drawback, by  applying
 the recent alternative  numerically stable algorithms of \cite{P15}
 and \cite{Pa},
whose Boolean cost matches the one of  \cite{MB72}
as long as the outputs
are required within the relative approximation error  bound $1/2^b$,
for $b=O(\log(n)$. This is certainly the case at the initial stage of our 
algorithm, when we compute the sign of $p(x)$ at $2n$ points.
 
We organize our presentation as follows.
In the next section we cover 
 some auxiliary results. In Section \ref{srresr}
%and  \ref{sbool} 
we recall the Boolean complexity estimate for the
approximation of the root radii.
In Section \ref{inclrr} we describe 
our real root algorithm.
In Section \ref{swex} we demonstrate
it by working example.
In Section \ref{scrisl} we extended 
our real root-finding algorithm 
to isolation of complex roots and 
root clusters.
Section \ref{stst},
the contribution of the second author,
covers
 the results of our numerical tests.
In Section \ref{sconc} we  briefly comment on 
further extension of our work.

%------------------------------------------------------------------------------

\section{Some Definitions and Auxiliary Results}\label{sdef}

Hereafter ``flop" stands for ``arithmetic operation".
``$\lg$" stands for ``$\log_2$".

 $O_B(\cdot)$ and $\tilde O_B(\cdot)$ 
denote the  Boolean complexity
up to some constant and poly-logarithmic factors, respectively.

$\tau$ is the overall bit-size of the coefficients of a polynomial $p(x)$ of 
 (\ref{eqpoly}). 

%------------------------------------------------------------------------------

%\subsection{Some Auxiliary Results
%Maps of the Variables and the Roots
%}\label{smpgvr}

\begin{theorem}\label{thdec} 
Count the roots of
a polynomial with their multiplicity.
Then a  polynomial $p(x)$ has an odd number of roots
in the real line interval $(\alpha,\beta)$
if and only if
$p(\alpha)p(\beta)<0$.
\end{theorem}

%------------------------------------------------------------------------------

%Some basic maps of  polynomial roots can be computed
%at a linear or nearly linear arithmetic cost.
 
%\begin{theorem}\label{thshsc} ({\em  Root Inversion, Shift and Scaling}.)

%(i) 
%This bound decreases to $2n-1$ multiplications if $b=0$.  

%(ii) Reversing a polynomial inverts all its roots involving no flops,
%that is, $p_{\rm rev}(x)=x^np(1/x)=\sum_{i=0}^np_ix^{n-i}=p_n\prod_{j=1}^n(1-xx_j)$.
%\end{theorem}

\begin{definition}\label{defisol}
 $D(z,\rho)=\{x:~|x-z|\le \rho\}$
denotes the closed disc with a complex center $z$ and a radius $\rho$.
 Such a disc 
 is $\gamma$-{\em isolated},
for $\gamma>1$,  
if the disc $D(z,\gamma \rho)$ 
 contains 
no other roots of the polynomial $p(x)$ of equation (\ref{eqpoly}).
A root $x_j$ is   $\gamma$-{\em isolated} if
the disc $D(x_j,(\gamma-1) |x_j|)$
contains no other roots of 
 the polynomial $p(x)$ besides $x_j$. 
% Define  
%the {\em isolation ratio} of  the disc  $D(X,r)$
%by maximizing $\gamma$.
%We say that a root $x$ of a polynomial is $\sigma$-{\em isolated
%relative to a complex point} $c$
%for $\sigma>0$ if $|y-x|>\sigma |x-c|$. If $c=0$ we call
%this root just $\sigma$-{\em isolated}. 
  \end{definition}

%------------------------------------------------------------------------------

The following theorem states that
Newton's iterations
\begin{equation}\label{eqnewt}
y_0=z,~y^{(h+1)}=y^{(h)}-p(y^{(h)})/p'(y^{(h)}),~h=0,1,\dots
\end{equation} 
converge  
with quadratic rate globally, that is, right from the start,
if they are initialized in a $3(n-1)$-isolated disc containing a single simple root
of a polynomial $p(x)$ of (\ref{eqpoly}).

%------------------------------------------------------------------------------

\begin{theorem}\label{thren} 
 Assume Newton's
  iteration (\ref{eqnewt}) for a polynomial $p=p(x)$ of  (\ref{eqpoly})
and let $0<3(n-1)|y_0-x_1|<|y_0-x_j|$ for $j=2,\dots,n$.
  Then  
   $|y_k-x_1|\le 2|y_0-x_1|/2^{2^k}$
for $k=0,1,\dots$.
\end{theorem}
\begin{proof} 
This is \cite[Theorem 2.4]{T98}, which strengthens 
\cite[Corollary~4.5]{R87}.
\end{proof} 

%------------------------------------------------------------------------------

The following two theorems  state some upper bounds on
 the Boolean cost of certain fundamental polynomial computations.

\begin{theorem}\label{thmltip} {\rm Multi-point Polynomial Evaluation.}  
Assume a real $b\ge 1$, a polynomial  $p(x)$ of (\ref{eqpoly}),
 and $k$ complex points $z_1,\dots,z_k$ such that  $k\ge n$.
Write $l=\lg(\max_{j=1}^n|z_j|)$.
 Then,  
at the Boolean cost $\tilde O_B((b+\tau)k+ lkn)$,
one can compute the values $v_j$ such that 
$|v_j-p(z_j)|\le 1/2^b$ for 
 $j=1,\dots,k$.
\end{theorem}

This is \cite[Theorem 3.9]{K98}, also proved in  \cite{vdH08}, \cite{KSa}, and \cite{PT14a}
(cf.  \cite[Lemma~21]{PT14a}.
All proofs boil down to estimating the Boolean cost of the algorithm of \cite{MB72},
which relies
on recursive polynomial division techniques of \cite{F72} and fast  FFT-based  
polynomial division of \cite{S72}. 

\begin{remark}
  \label{renewt}
Clearly, the same asymptotic bound applies to the Boolean complexity of 
a single Newton's iteration 
performed concurrently at $m$ points.
\end{remark}
%Let us extend Theorem \ref{thmltip} to 
%estimating the overall Boolean complexity of  concurrent application of Newton's iterations 
%initialized at $m$ distinct points.
%At first recall the following  
%\begin{theorem}
 % \label{prop:anwt-complexity}
 % Let $p \in \CC[x]$ be of degree $n$ 
 % such that $\normi{p} \leq 2^{\tau}$.
 % Assume that we are given a $\lambda$-approximation of $p$, $\wt{p}$, such
 % that $\normi{p - \wt{p}} \leq 2^{-\lambda}$, where $\lambda = \ell + L + \OO(n\tau)$, $\ell > 0$.

%(i)  Then the maximum number of bits needed by Newton iterations (\ref{eqnewt}) is
%  $\sO(L + n\tau + \ell)$, and the overall complexity of the Newton step is 
%  $\sOB(n^2\tau + nL + n\ell)$.

%(ii) The same asymptotic bound holds for Newton's iterations applied concurrently
%at $m$ points.
%\end{theorem}
%Part (i) is a slight modification of Lemma~10 and Remark~11 in \cite{PT13,PT15}, 
%see also \cite{PT14b}.

%Part (ii) follows because each Newton iterations step consists of the evaluation of the polynomial $p$ and its derivative and because we
%can perform all these operations simultaneously 
%by using multipoint evaluation  at the same 
%asymptotic Boolean cost (cf. our Theorem \ref{thmltip} and \cite[Theorem~14]{PT15}). 

\begin{theorem}\label{thrrrf} 
Given a real $b\ge 1$,
 a polynomial  $p(x)$ of (\ref{eqpoly}), and
the complex centers $z_j$ and the radii $\rho_j$ of $m$  discs 
$D_j=D(z_j,\rho_j)$, $j=1,\dots,m$,
each
covering a single  simple  root of 
the polynomial  $p(x)$, 
write $L=\max_{j=1}^m \lg(|z_j|+\rho_j)$.
Then, at the Boolean cost $\tilde O_B((b+L)n+\tau n^2)$,
one can compute $m$
new inclusion discs for the same roots, with 
the radii  decreased by factors of
at least $2^b$ provided that  

(i) all the $m$ roots and all the $m$ centers $z_1, \dots,z_m$ are real or

(ii) all the $m$ discs $D_1,\dots,D_m$ are $\gamma$-isolated for a constant $\gamma>1$.
\end{theorem}

Part (i) has been proved in \cite{PT13} based on the algorithms of 
\cite{PL99} and \cite{PMRQT07} (cf. also  \cite{PT15} and \cite{SM15}). 
Part (ii) has been proved in \cite{PT14b} based on the
approximation of the power sums of the roots of a polynomial.

%------------------------------------------------------------------------------

\section{Root Radii and Their Estimation}\label{srresr}

%------------------------------------------------------------------------------

\begin{definition}\label{defrr}
List the absolute values of the roots of $p(x)$ in the 
non-increasing order, denote them  $r_j=|x_j|$ for $j=1,\dots,n$,
 $r_1\ge r_2\ge\cdots\ge r_n$, and call them the root radii of 
 the polynomial $p(x)$. 
  \end{definition}
The following result bounds the largest root radius $r_1$.

%------------------------------------------------------------------------------

\begin{theorem}\label{thextrrrd} (See \cite{VdS70}.)
For a polynomial $p(x)$ of  (\ref{eqpoly}) and
$r_1=\max_{j=1}^{n}|x_j|$, it holds that
\begin{equation}\label{eqr1}
0.5 r_1^+/n\le r_1\le r_1^+~{\rm for}~r_1^+=2\max_{i=1}^{n}|p_{n-i}/p_n|.
\end{equation}
%(One can compute the values $\gamma^+$  by using $n$
%divisions and $n-1$ comparisons.)
\end{theorem}

\begin{remark}\label{reprr} 
The theorem bounds all the root radii, 
but the paper \cite{S05} (extending the
previous work in \cite{L798}, \cite{K86},  
and \cite{H98}) presents stronger 
upper estimates for the positive roots of a polynomial $p(x)$.
By applying these estimates to the reverse polynomial 
$p_{\rm rev}(x)=x^np(1/x)$
and to the polynomials $p(-x)$ and $p_{\rm rev}(-x)$,
one can bound the positive roots from below 
and the negative roots from both below and above. 
\end{remark}

\begin{theorem}\label{thrr} 
 Assume a polynomial $p=p(x)$ of  (\ref{eqpoly})
and a positive $\Delta$.
 Then, within the Boolean cost bound $\tilde O_B(\tau n^2)$,
one can compute approximations $\tilde r_j$
to all root radii $r_j$
such that $1/\Delta\le \tilde r_j/r_j\le \Delta$ 
for
 $j=1,\dots,n$, provided that 
$\lg(\frac{1}{1-\Delta})=O(\lg(n))$,
that is, 
$|\tilde r_j/r_j-1|\le c/n^d$
for a fixed pair of constants $c>0$ and $d$. 
\end{theorem}

This is 
 \cite[Corollary 14.3]{S82}. At first the root radii are approximated 
at a dominated cost in the case of $\Delta=2n$.
See some details of the algorithm in
 \cite[Section 4]{P00}. 
 \cite[Section 14]{S82} 
and  \cite[Section 4]{P00}
 cite the related works 
 \cite{O40}, \cite{G72}, \cite{H74},  \cite{A85}, and one   
can also compare  the relevant techniques of the power geometry
in  \cite{B79}
and \cite{B98}, developed for
 the study of 
algebraic and differential equations.

In order to extend Theorem \ref{thrr}  
to the case of $\Delta=(2n)^{1/2^k}$ for any positive integer $k$, 
at first apply $k$  Dandelin's
 root-squaring iterations 
 to the monic polynomial $q_0(x)=p(x)/p_n$
(cf. \cite{H59}),
that is, compute recursively the polynomials
\begin{equation}\label{eqdnd}
 q_{i}(x)=(-1)^nq_{i-1}(\sqrt{x}~)q_{i-1}(-\sqrt{x}~)=\prod_{j=1}^n(x-x_{j}^{2^i}),
~{\rm for}~i=1,2,\dots
\end{equation}
Then 
approximate 
the root radii $r_j^{(k)}$ of $q_k(x)$ by
applying Theorem \ref{thrr} for $\Delta=2n$ and 
for $p(x)$ replaced by 
$q_{k}(x)$.
Finally approximate 
the root radii $r_j$ of $p(x)$
as $r_j=(r_j^{(k)})^{1/2^k}$.    
This enables us to decrease $\Delta$ from $2n$ to at most 
$1+c/n^d=1+2^{O(\lg (n))}$ for any 
fixed pair of constants $c>0$ and $d$
by using $k=O(\lg (n))$ Dandelin's iterations.
The cost bound of  Theorem \ref{thrr} follows.  

%------------------------------------------------------------------------------

\begin{remark}\label{relg}
By using $k=\lceil\ln(s\lg(n))\rceil$
Dandelin's root-squaring iterations, for $s>1$,
the same algorithm supporting 
 Theorem \ref{thrr} approximates 
the root radii within the relative error bound $\frac{1}{s\ln(2n)}$.
\end{remark}

%------------------------------------------------------------------------------

\begin{corollary}\label{corr} 
 Assume a real $b\ge 1$, a polynomial $p(x)$ of  (\ref{eqpoly}),
and a complex $z$.
Then, within the Boolean cost bound $\tilde O_B((\tau +n(1+\beta))n^2)$,
for $\beta=\lg(2+|z|)$,
one can compute approximations $\tilde  r_j\approx \bar r_j$
to the distances $\bar r_j=|z-x_j|$ from the point $z$ to all roots $x_j$
of the polynomial $p(x)$
such that $1/\Delta\le \tilde r_j/\bar r_j\le \Delta$, 
for  
 $j=1,\dots,n$, 
provided that $\lg(\frac{1}{\Delta-1})=O(\lg(n))$.
\end{corollary} 
\begin{proof}
Given a polynomial $p(x)$ of (\ref{eqpoly}) and a 
complex scalar $z$, one can
compute the coefficients of the polynomial $q(x)=p(x+z)$ by using
$O(n \lg (n))$ flops (cf. \cite{P01}). 
The root radii of the polynomial $q(x)$ for
a complex scalar $z$
 are equal to the distances $|x_j-z|$ from the point $z$ 
to the roots $x_j$ of $p(x)$. 

Now let  $\tau_q$ denote the bit-size of the coefficients
of $q(x)$,
and let
$\bar r_j$ for $j=1,\dots,n$ denote its root radii,
listed in the non-increasing order. 
Then, clearly, 
$\bar r_j\le r_j+|z|$ for $j=1,\dots,n$.
Furthermore represent the coefficients of the polynomial $q(x)$
with the bit-size $\tau+n$ by using representation of $p(x)$
with  bit-size
$\tilde O(\tau +n(1+ \beta))$
for $\beta=\lg(2+|z|)$.
By applying Theorem \ref{thrr} 
to the polynomial $q(x)$,  extend the cost bounds 
from the root radii to the distances.
\end{proof}

%------------------------------------------------------------------------------

\section{Isolation and Approximation of Simple and Well-Con\-di\-tioned Real Roots}\label{inclrr}

\begin{algorithm}\label{algrrinc} {\em  Approximation of simple and well-conditioned real roots.}

\medskip

\item{\textsc{Input:}} two positive integers $n$ and $r$
 such that $0<r<n$, a positive tolerance bound $t$, 
three real constants $b$, $c$ and $d$
such that $b\ge 1$ and $c>0$,
    the coefficients of a polynomial $p(x)$ of equation (\ref{eqpoly}), and
an upper estimate $r'$ for the unknown
number of its real roots, 
$r'\le n$.

\medskip

    \item{\textsc{Output:}}
 Approximations within the relative  error bound $1/2^b$ to some or all
real roots $x_1,\dots,x_{r_+}$ of the polynomial $p(x)$, including all 
real roots that are both
single and $(1+c/n^d)$-isolated.

\medskip

\item{\textsc{Computations:}}
         \begin{enumerate}
\item%1
 Compute 
approximations $\tilde r_1,\dots,\tilde r_n$
to the  root radii of a polynomial $p(x)$ of (\ref{eqpoly}) 
within the relative error bound $c/n^{d}$,
$\tilde r_1\le \cdots \le\tilde r_n$ 
(see Theorem \ref{thrr}). (Each approximation $\tilde r_i$ defines 
at most two  real line intervals  
that can include real roots lying near $\pm\tilde r_i$.
Overall we obtain 
$2n$ candidate inclusion intervals
$\mathcal I_1,\dots,\mathcal I_{2n}$, some of which can overlap or coincide
with each other.) 

\item%2
Define $\mathbb T=\{t_1,\dots,t_{n'}\}$, 
the set of the distinct endpoints
of these intervals in non-decreasing order, for $n'\le 4n$.
Compute
 the sign of the polynomial $p(x)$ on this set. 
If the sign changes at two consecutive real points, 
$t'=t_{j}$ and $t''=t_{j+1}$,
select the line interval $(t',t'')$.

\item%3
Apply the algorithm of \cite{PT13}, \cite{PT15}, which
supports part (i) of Theorem \ref{thrrrf},  concurrently 
to all selected  line intervals.
As soon as the algorithm decreases the length of a
selected interval to or below $1/2^{b-1}$,
output the midpoint as an approximation to a real root. 
If $r'$ such approximations have been output
or if the cost of performing the iterations of 
the algorithm of \cite{PT13}, \cite{PT15}
reaches the tolerance bound $t$, then stop the computations.
(At stopping, there can remain some selected intervals of length 
exceeding $1/2^b$. We can output them with the label ``intervals with real roots, apparently 
ill-con\-di\-tioned".) 
  \end{enumerate}
\end{algorithm}

Compare the error bound on the root radii at Stage 1 with 
the definition of a single $(1+c/n^d)$-isolated real root of $p(x)$
in Definition \ref{defisol} and conclude that 
one of the inclusion intervals 
$\mathcal I_j$, for $1\le j\le 2n$, contains this root
and no other roots of  $p(x)$. Now the
correctness of the algorithm follows from Theorem
\ref{thdec}.

The Boolean cost  at Stage 1 is $\tilde O_B((\tau+l) n^2)$,
for $l=\max_{j=1}^m \lg(|x_j|)$ (by virtue of 
 Theorem \ref{thrr}),
at Stage 2 is  $\tilde O_B(\tau n+l n^2)$ 
(by virtue of
 Theorem \ref{thmltip}, for $b=2$),
and at Stage 3 is  $\tilde O_B((\tau+l) n^2+bn)$
(by virtue of part (i) of
 Theorem \ref{thrrrf} for $L=O(l))$.
In view of (\ref{eqr1}), we can write $l=O(\tau)$,
and then the overall asymptotic cost bound,  dominated at Stage 3,
turns into  $\tilde O_B(\tau n^2+bn)$.
This is the same cost bound as at Stage 1 if $b=O(\tau n)$
and is the same as  Stage 2 if in addition $\tau=O(l)$.
 Summarizing we obtain the following estimate.

%------------------------------------------------------------------------------

\begin{theorem}\label{thrrrr}
Suppose that we are
given  the coefficients of a polynomial $p(x)$ of equation (\ref{eqpoly})
and three real constants $b$, $c$, and $d$ such that
$b\ge 1$ and $c>0$.
Then, at the cost $\tilde O_B(\tau n^2+bn)$,
 we can approximate
all the  $(1+c/n^d)$-isolated real roots  
of  $p(x)$  within the relative error bound $1/2^b$. 
\end{theorem}

%------------------------------------------------------------------------------

%The following natural  variations of Algorithm \ref{algrrinc}
%may enable its heuristic simplifications.

%------------------------------------------------------------------------------

\begin{remark}\label{repm}
Before we begin the computations of Stage 2 of 
 of Algorithm \ref{algrrinc}, we can
 narrow the ranges for positive and negative  roots 
of the polynomial  $p(x)$  by 
following the recipes of Remark \ref{reprr}.
Then at  Stage 2 we can avoid the
evaluation of the polynomial at the points
lying outside these ranges.
\end{remark}

\begin{remark}\label{reheur}
Before we started Stage 3, we 
 can recompute the root radii
for a smaller value $c/n^d$ and then
check if some intervals,  selected under this value,
 lie strictly inside the old ones,
sharing no endpoints,
and hence are $\gamma$-isolated for some $\gamma>1$.
To such intervals we can 
apply a simplified version of the algorithm of  \cite{PT13},
\cite{PT15},  which
  combines bisection and Newton's iteration
and does not involve the more advanced algorithm 
of double exponential sieve
(also called the bisection of the exponent).
\end{remark}

%------------------------------------------------------------------------------

\begin{remark}\label{reprec}
For multi-point evaluation of the polynomial  $p(x)$,
we can apply the algorithm of \cite{MB72} 
with extended precision, but 
if the output is only required with a low precision,
then instead we can apply
 the algorithms of \cite{P15} and \cite{Pa} 
  with double precision
at the same asymptotic Boolean cost.
In particular at Stage 2 of  Algorithm \ref{algrrinc}
 we only seek
 a single bit (of the sign of $p(x)$) per point of
evaluation, and at Stage 3  at the initial Newton's
iterations,  crude values of $p(x)$ can be sufficient.
At Stage 2 of contracting $r$ suspect real intervals,
we perform order of $O(r\lg (b))$ evaluations, and 
if $r$ is small, 
which is quite typically the case 
in the applications to algebraic and geometric computations,
then it can be non-costly even to apply
the so called Horner's algorithm. It uses $2nr$
arithmetic operations, that is, less than the
algorithms of \cite {MB72},  \cite {P15}, and  \cite {Pa},
if $r=o(\log^2(n))$.
\end{remark}

%------------------------------------------------------------------------------

\begin{remark}\label{revar}  
If we only need to  
  approximate the very well isolated
real roots, we can 
apply Algorithm \ref{algrrinc} at a lower cost 
for a larger isolation ratio.
If we do not know how well the roots are isolated,
we can at first apply Algorithm \ref{algrrinc}
assuming a larger ratio
and, if the algorithm fails, re-apply it recursively,
eah time assuming stronger isolation.
\end{remark}

%------------------------------------------------------------------------------

\section{Working Example}\label{swex}

%------------------------------------------------------------------------------

Consider the following polynomial
$$
p(x) = 8x^7 + 16x^6 + 16x^5 + 16x^4 - 23x^3 -30x^2 + 3x + 4.
$$
This  product of the 4th degree Chebyshev polynomial 
of the first kind and the cubic polynomial $x^3 + 2x^2 + 3x + 4$ 
has five real roots and two complex conjugate non-real roots. All seven roots are well-separated, namely,

$$
   0.3827 + 0.0000i,
  -0.3827 + 0.0000i,
   0.9239 + 0.0000i,
  -0.9239 + 0.0000i,
$$
$$
  -0.1747 + 1.5469i,
  -0.1747 - 1.5469i,
  -1.6506 + 0.0000i.
$$

 It took 14 root-squaring iterations to  estimate the
seven root-radii with a precision of at least 0.001:
$$
[0.3827, 0.3827], 
[0.3826, 0.3828], 
[0.9237, 0.9240], 
[0.9238, 0.9241],
$$
$$ 
[1.5565, 1.5570], 
[1.5565, 1.5570], 
[1.6504, 1.6507]. 
$$

The range $[1.5565, 1.5570]$ was for two roots.
Taking  the negations into account, we obtained 12 
intervals containing all real roots,
among which we  counted the intervals $[-1.5565, -1.5570]$ and
$[1.5565, 1.5570]$ twice:
$$
[-1.6507, -1.6504], 
[-1.5570, -1.5565], 
[-0.9241, -0.9238],
$$
$$ 
[-0.9240, -0.9237], 
[-0.3828, -0.3826], 
[-0.3827, -0.3827], 
[0.3827, 0.3827],
$$
$$ 
[0.3826, 0.3828], 
[0.9237, 0.9240], 
[0.9238, 0.9241], 
$$
$$
[1.5565, 1.5570], 
[1.6504, 1.6507].
$$

%Note that these intervals are not overlapping. If it is not the case, the root-radii algorithm will continue to refine the estimations until the estimation intervals are well-separated. 

Some of the ranges for the remaining root radii overlapped pairwise,
but our algorithm evaluated  the polynomial $p(x)$ at the endpoints of
all the twelve intervals. 
Wherever there was the change of the sign, 
we concluded that the interval contains a root of  the polynomial $p(x)$. 
We obtained nine (rather  than five or even seven)
  intervals with the sign changes,
namely:
$$
[-1.6507, -1.6504] 
[-0.9241, -0.9238] 
[-0.9240, -0.9237] 
[-0.3828, -0.3826] 
$$
$$ 
[-0.3827, -0.3827]
[0.3827, 0.3827] 
[0.3826, 0.3828] 
[0.9237, 0.9240] 
[0.9238, 0.9241]
$$

Then our algorithm initiated Newton's iteration at the mid-point of each of these intervals. Whenever the Newton's map fell outside the interval, 
the algorithm switched to bisection method instead
of Newton's. 
%Because Newton's iteration is locally quadratically convergent, 
Very soon the algorithm converged to all real roots of $p(x)$.

The first Newton's iteration produced the following
real root estimates: 
$$
-1.65062921,
-0.92387954,
-0.92387954,
-0.38268343,
$$
$$
-0.38268343,
0.38268343,
0.38268343,
0.92387954,
0.92387954.
$$

The second iteration produced the estimates
$$
-1.65062919,
-0.92387953,
-0.92387953,
-0.38268343,
-0.38268343,
$$
$$
0.38268343,
0.38268343,
0.92387953,
0.92387953.
$$

The maximum difference of the estimates output by the second and the third iterations was less than $10^{-8}$. Thus the algorithm stopped and returned the values of five distinct real roots of the polynomial $p(x)$:
$$
-1.65062919,
-0.92387953,
-0.38268343,
0.38268343,
0.92387953.
$$

%------------------------------------------------------------------------------
%------------------------------------------------------------------------------

\section{Isolation of Simple and Well-Conditioned Complex Roots by Means of  
Root-radii Approximation}\label{scrisl}

%------------------------------------------------------------------------------

\begin{algorithm}\label{algisol} 
{\em Isolation of Simple and Well-Conditioned Complex Roots.}
\item{\textsc{Input:}} two positive scalars $\rho$ and $\epsilon$
  and  the coefficients of a polynomial $p(x)$ of  (\ref{eqpoly}).

    \item{\textsc{Output:}} 
 Set of approximations to the roots of
the polynomial $p(x)$ within $\rho\sqrt 2$
such that with probability at least 
$1-\epsilon$ this set includes all
roots having no other roots of the
polynomial $p(x)$ in their 
$\Delta'$-neighborhoods for
$\Delta'=(22.63~n^4+2\epsilon)\rho/\epsilon$.

\item{\textsc{Initialization:}} 
Fix  
a reasonably large scalar $\eta$, say, $\eta=100$.
%Write  $c=6/(\eta+1)$ and $d=h+3$.
Generate a  random value
 $\phi$  under the uniform probability distribution
in the range $[\pi/8,3\pi/8]$.

\item{\textsc{Computations:}}
         \begin{enumerate}
         \item%1
(Three Shifts of the Variable.)
Compute the value
$r_1^+=2\max_{i=1}^{n}|p_{n-i}/p_n|$ of (\ref{eqr1}).
 Then compute the coefficients of the three polynomials
$q(x)=p(x-\eta r_1^+)$,
$q_-(x)=p(x-\eta r_1^+\sqrt {-1})$, and
$q_{\phi}(x)=p(x-\eta r_1^+\exp(\phi \sqrt {-1}))$.

\item%2
 Compute %some
approximations  
to all the $n$ root radii of each of these  three polynomials 
%$q(x)$, $q_-(x)$, and $ q_{\phi}(x)$ 
within the relative error bound $\rho/((r_1^++1)\eta)$. 
[This defines three  families of large thin annuli.
 Each family consists of
$n$ annuli and  
 each annulus contains a 
 root of  $p(x)$. Multiple roots define multiple annuli. Clusters 
of roots define  clusters 
of overlapping annuli.]

\item%3 
Compute the intersections of all pairs of  the annuli
from the first two families and of the disc $D(0,r_1^+)$.
[We only care about the roots of $p(x)$, 
and all of them lie in the disc $D(0,r_1^+)$. 
The intersection of each annulus  
with the disc $D(0,r_1^+)$ is close to a rectangular 
(vertical or horizontal) band of width 
at most $\rho$ on the complex plane
%is close to a rectangle, which in turn is close to a square, 
because every annulus has  width at most
$\rho\le (r_1^++1)/\eta$ where we have 
chosen $\eta$ large enough.
The intersection of any pair of annuli 
from the two families is close to a square,  
with nearly vertical and nearly horizontal edges
%, or are directed by the angle $\phi +\pi/2$
of length at most $\rho$]. 
%such that
%\begin{equation}\label{eqrho}
%0.1(\eta+1) c r_1^+\sqrt 2/n^{d+1}<\rho=0.2 (\eta+1) c r_1^+/n^{d+1}=1.2r_1^+/n^{h+4}.
%\end{equation}
The disc $D(0,r_1^+)$ contains 
a {\em grid} made up of 
$n^2$ such squares, said to be {\em nodes}. [Some of them can be  
merged together (in the case of multiple roots) or nearly merged 
(in the case of root clusters).]

\item%4 
For each annulus of the third family,   
determine its intersection with the nodes 
of the grid, and if the annulus intersects
only a single node, then output its center
 and stop. 
  \end{enumerate}
\end{algorithm}

%------------------------------------------------------------------------------

Let us prove {\em correctness of the algorithm}. 
 At first
readily verify the following lemma.

\begin{lemma}\label{lemdsc}
Suppose that a line passes through a disc $D(z,\rho')$
in the direction chosen at random under the 
uniform probability distribution of its angle
in the range $[\alpha,\alpha+\gamma]$
for $0\le \gamma/(2\pi) \le 1$.
Then the line intersects a disc $D(z',\rho')$
with a probability at most $P=\frac{2}{\gamma}\arctan (\frac{\rho'}{|z-z'|})$.
\end{lemma}
 
Next apply the lemma to 
 a pair of nodes of the grid
with distance $\Delta=|z-z'|$ between their two centers $z$ and $z'$.
In this case  $\gamma=1/8$ and the two nodes lie in the two discs $D(z,\rho')$ 
and  $D(z',\rho')$ for $\rho'=\rho \sqrt 2$.
Furthermore assume that $\Delta\gg \rho'$,  so that
$\frac{\rho'}{|z-z'|}\approx\arctan (\frac{\rho'}{|z-z'|})
\le 1.001 \frac{\rho'}{|z-z'|}$.
Then the lemma implies that
\begin{equation}\label{eqprob}
 P\le 16\sqrt 2 \arctan (\rho'/\Delta)<22.6275~\rho/\Delta.
\end{equation}

\begin{theorem}\label{thmdsc}
Let the grid of Algorithm \ref{algisol}
have $N_{\rho}$ nodes overall, $N_{\rho}\le n^2$.
Suppose that an annulus of the third family
 intersects a node whose center  has
no centers of the other nodes of the grid
within the distance $\Delta=22.63 (N_{\rho}-1)\rho/\epsilon$,
for a fixed positive tolerance $\epsilon$. Then this
annulus intersects another node of the grid
with a probability at most $\epsilon$. 
\end{theorem}
\begin{proof}
Apply bound (\ref{eqprob}) to all pairs of the nodes of the grid.
\end{proof}

Now correctness of the algorithm follows because 
every root of the polynomial $p(x)$ lies in some 
 annulus of each of the three families. 

The Boolean complexity bounds for Stage 1 of Algorithm
\ref{algrrinc} are immediately extended to 
cover the Boolean complexity of Stages 1 and 2 
of Algorithm \ref{algisol} as long as 
$\log((r_1^++1)\eta /\rho=O(\log(n))$
(cf. Corollary \ref{corr}).
The Boolean complexity at Stage 3 of Algorithm
\ref{algrrinc} is dominated if we 
check separately whether the real parts and the
imaginary parts of nodes overlap  pairwise. 
Likewise the Boolean complexity at Stage 3 of Algorithm
\ref{algrrinc} is dominated if, for every annuls  
of the third family, we
apply the bisection process, which 
begins with the 
median knot of the grid and recursively
discards 
the knots that lie above and below the annulus.

\begin{remark}\label{reclst}
We can modify Stage 2 of Algorithm
\ref{algrrinc} by collapsing every chain of $m$
pairwise overlapping or coinciding root radii
intervals (where $m\le n$) into a single interval of relative width
at most $m\rho/((r_1^++1)\eta)$ and by assigning to this 
interval multiplicity $m$.
Such extended root radii define thicker 
annuli, of multiplicity $m>1$.
The intersection of two annuli of two families, 
having multiplicity $m_1$ and $m_1$, respectively,
defines a node of the new grid, to which we assign   
multiplicity $\min\{m_1,m_2\}$.  
We do not change  Stage 4, except that 
an output node of
multiplicity $m$
contains exactly $m$ 
roots of the polynomial  $p(x)$, each
counted according to its multiplicity.
\end{remark}

%------------------------------------------------------------------------------

\begin{remark}\label{rerefin} 
Clearly the assumptions of part (i) of Theorem \ref{thrrrf}
are satisfied for the centers of the output nodes of  Algorithm
\ref{algrrinc}, and so we can extend the algorithm to
refining the approximations to the well conditioned roots
of the polynomial  $p(x)$ given by these centers.
For 
every node output by the algorithm of the previous remark,
we can compress its superscribing disc fast,
according to part (ii) of Theorem \ref{thrrrf}, 
as long the assumptions of that part
are satisfied.
\end{remark}

%------------------------------------------------------------------------------

\begin{remark}\label{reshfts}
We have chosen absolutely large shift values 
in order to simplify our analysis, although in
this case we must use higher computational precision. 
One can apply 
 the same algorithm for smaller value of $\eta$.
This heuristic variation may work and 
may allow to use lower precision of computing. 
\end{remark}

%------------------------------------------------------------------------------

\begin{remark}\label{reminangle}  
If the nodes were points and the annuli straight lines, 
one could compute the minimum nonzero
angle $\phi$ between all the leftmost and all the rightmost nodes of the grid
lying above.
Then by directing the third family of the annuli 
by the angle $\phi/2$, one could ensure deterministically 
that every annulus passes through exactly a single node 
 of the grid. This recipe 
can be extended to the case of sufficiently
thin annuli and small nodes. 
If the width of the annuli is too large, 
we can decrease it 
by using more root-squaring iterations, but 
it is not clear for which class of inputs  
we can still keep the claimed overall cost bound
under this recipe. 
\end{remark}

%------------------------------------------------------------------------------
%------------------------------------------------------------------------------

%------------------------------------------------------------------------------

\section{The Tests for Real Root-finding with Algorithm \ref{algrrinc}}\label{stst}
 
%------------------------------------------------------------------------------

We tested a variant of
our isolation Algorithm \ref{algrrinc} described in Remarks 
\ref{reheur}--\ref{revar}. We applied it 
  to polynomials of degrees $n=64$,$128,256,512$,$1024$ having no clustered roots.
In this variant we simplified Stage 3:
instead of the algorithm of \cite{PT13}, \cite{PT15}
we applied three Newton's iterations, initialized at the center of the 
isolation interval
output by Algorithm \ref{algrrinc}. In the case of poor convergence,
we performed a few bisections of the output interval of 
 Algorithm \ref{algrrinc} and then applied three Newton's iterations again.
We performed  all computations with  
 double precision and estimated the output 
errors by comparing  our results with the outputs 
of MATLAB function "roots()". 
%the root-squaring procedure was achieved by taking convolution of $p(\sqrt{x})$ and $p(-\sqrt{x})$.

We generated our input polynomials as the products of the Chebyshev polynomials of the first kind
(having degree $r$ and thus having $r$ simple real roots) with polynomials of degree $n-r$
of the following three families,

1. Polynomials with real standard Gaussian random  coefficients.

2. Polynomials with complex standard Gaussian random coefficients.

3. Polynomials with integral consecutive coefficients starting at 1, i.e., $p(x) = 1 + 2x + ... + (n-1)x^n$.

Table \ref{RealRoots} displays the number of Dandelin's root-squaring iterations 
(in the column ``Iter")
 and the maximum error (in the column ``Error"). The tests showed very slow growth 
of the number of iterations, and consequently of the computational cost of Stage 1,
 as $n$ increased from 64 to 1024.

%Our code is available upon request. 

\begin{table}[h]
\caption{Real Root-finding with Using Root Radii Estimation}
\label{RealRoots}
\begin{center}
\begin{tabular}{|*{8}{c|}}
\hline
\multicolumn{2}{|c|}{}&\multicolumn{2}{c|}{Type 1} 	& \multicolumn{2}{c|}{Type 2} 	& \multicolumn{2}{c|}{Type 3} \\ \hline
\bf {n}	&	\bf {r}	&	Iter	&	Error	&	Iter	&	Error	&	Iter	&	Error	\\ \hline
64	&	4	&	4	&	6.42E-11	&	4	&	6.42E-11	&	4	&	6.42E-11	\\ \hline
64	&	8	&	6	&	3.16E-05	&	6	&	3.16E-05	&	6	&	3.16E-05	\\ \hline
64	&	12	&	7	&	4.36E-03	&	7	&	4.36E-03	&	7	&	4.36E-03	\\ \hline
128	&	4	&	4	&	6.42E-11	&	4	&	6.42E-11	&	4	&	6.42E-11	\\ \hline
128	&	8	&	6	&	3.16E-05	&	6	&	3.16E-05	&	6	&	3.16E-05	\\ \hline
128	&	12	&	7	&	4.36E-03	&	7	&	4.36E-03	&	7	&	4.36E-03	\\ \hline
256	&	4	&	4	&	6.42E-11	&	4	&	6.42E-11	&	4	&	6.42E-11	\\ \hline
256	&	8	&	6	&	3.16E-05	&	6	&	3.16E-05	&	6	&	3.16E-05	\\ \hline
256	&	12	&	7	&	4.36E-03	&	7	&	4.36E-03	&	7	&	4.36E-03	\\ \hline
512	&	4	&	4	&	6.42E-11	&	4	&	6.42E-11	&	4	&	6.42E-11	\\ \hline
512	&	8	&	6	&	3.16E-05	&	6	&	3.16E-05	&	6	&	3.16E-05	\\ \hline
512	&	12	&	7	&	4.36E-03	&	7	&	4.36E-03	&	7	&	4.36E-03	\\ \hline
1024	&	4	&	4	&	6.42E-11	&	4	&	6.42E-11	&	4	&	6.42E-11	\\ \hline
1024	&	8	&	6	&	3.16E-05	&	6	&	3.16E-05	&	6	&	3.16E-05	\\ \hline
1024	&	12	&	7	&	4.36E-03	&	7	&	4.36E-03	&	7	&	4.36E-03	\\ \hline

\end{tabular}
\end{center}
\end{table}

%\clearpage

\section{Conclusions}\label{sconc}

In our future study we plan to elaborate upon the extension
of  our isolation 
of multiple roots and root clusters, by  
 incorporating the techniques that extend part (ii) of 
Theorem \ref{thrrrf}
and complementing them by the algorithm of \cite{K98}. 
These techniques 
should enable us to  
 split out the factor of $p(x)$ whose root set is precisely the cluster
of the roots of the polynomial $p(x)$ lying in the output node,
and then we can work on root-finding for the remaining complementary factor.
Developing this approach, we
are going to explore potential advantages of shifting the
origin into the center of gravity of the roots, $-p_{n-1}/(np_n)$ (cf. \cite{S82}),
or into the roots of higher order derivatives of $p(x)$ (cf. \cite[Section 15]{MP13}).

We plan to test our current and extended
algorithms against the available real root-finders
and to explore the chances for enhancing the efficiency of our algorithms
by means of their combination
with
some of the other  highly efficient techniques 
 known for root-finding.

\medskip

{\bf Acknowledgements:}
This work has been supported by NSF Grant CCF 1116736 
and  PSC CUNY Award 67699-00 45. 
%We are also grateful to the
%reviewers for helpful comments.

%------------------------------------------------------------------------------
%------------------------------------------------------------------------------

\end{document}